\documentclass[preprint,12pt]{elsarticle}

\usepackage{amssymb}
\usepackage{amsmath}
\usepackage{amsthm}

\newtheorem{thm}{Theorem}[section]
\newtheorem{cor}[thm]{Corollary}
\newtheorem{lem}[thm]{Lemma}
\newtheorem{prop}[thm]{Proposition}

\theoremstyle{definition}
\newtheorem{defn}[thm]{Definition}
\theoremstyle{remark}
\newtheorem{rem}[thm]{Remark}
\newtheorem{exm}[thm]{Example}
\newtheorem{conj}[thm]{Conjecture}

\def\CC{\mathbb{C}}
\def\AAA{\mathbb{A}}

\def\SL{\mathrm{SL}}
\def\sht{\mathrm{Sht}}

\usepackage{mathtools, mathdots}
\DeclarePairedDelimiter\ceil{\lceil}{\rceil}
\DeclarePairedDelimiter\floor{\lfloor}{\rfloor}

\journal{Journal of Algebra}

\begin{document}

\begin{frontmatter}



\title{Euler characteristic of crepant resolutions of specific modular quotient singularities}


\author{Linghu Fan} 
\ead{linghu.fan@ipmu.jp}

\affiliation{organization={Graduate School of Mathematical Sciences, The University of Tokyo},
            addressline={3-8-1 Komaba, Meguro-ku}, 
            city={Tokyo},
            postcode={153-8914}, 
            country={Japan}}

\begin{abstract}
In this paper, we consider a generalization of the McKay correspondence in positive characteristic regarding the Euler characteristic of crepant resolutions of quotient singularities given by finite subgroups of the special linear group. As the main result, we prove that this generalization holds for groups with a specific semidirect product structure, using the wild McKay correspondence over finite fields as mass formulas. Furthermore, two additional examples with more complicated structures are also given. Based on our main result, we propose a conjectural form of the generalized McKay correspondence in the modular case.
\end{abstract}




\end{frontmatter}



\section{Introduction}
The Euler characteristic of crepant resolutions is one of the important geometric invariants in the study of the McKay correspondence. As a series of results connecting algebra with geometry, the McKay correspondence studies relations between algebraic properties of groups and geometric properties of the associated quotient singularities. Over the complex numbers, it was conjectured by Reid and proved by Batyrev via motivic integration (\cite{batyrev1999non}), that for a finite group $G\subseteq \SL(n,\CC)$ and the associated quotient singularity $X:=\CC^n/G$, assuming the existence of a crepant resolution $Y\to X$, the Euler characteristic $e(Y)$ equals the number of conjugacy classes $\#\mathrm{Conj}(G)$.

In this paper, we focus on the case where the base field has positive characteristic. Let $k$ be an algebraically closed field of prime characteristic $p>0$, and $G\subseteq \SL(n,k)$ be a finite group giving the quotient singularity $X:=\AAA^n_k/G$. For a smooth $k$-variety, its (topological) Euler characteristic is defined as the alternating sum of the Betti numbers given by the $l$-adic cohomology with compact support, where $l$ is a prime different from $p$.

It is natural to ask whether the equality between the Euler characteristic of crepant resolutions and the number of conjugacy classes still holds. If the order of $G$ is not divisible by $p$, then there is an affirmative answer known as the tame McKay correspondence (the reader can refer to \cite{wood2015mass}, Section 5.2). Unfortunately, in the modular case, where $\#G$ is divisible by $p$, the analog of Batyrev's theorem has been disproved by constructing counterexamples, such as Chen, Du and Gao's counterexample in the non-small case (\cite{chen2020modular}), Yamamoto's construction for the quotient variety given by the symmetric group $S_3\subseteq \SL(3)$ in characteristic $3$ (\cite{yamamoto2021crepant}) and the author's construction for the quotient variety given by the permutation action of the alternating group $A_4$ in characteristic $2$ (\cite{fan2023crepant}). 

On the other hand, there are also several known results implying that the analog of Batyrev's theorem in positive characteristic holds for some specific modular quotient singularities. We list them below.

\begin{thm}[Known results]\label{HCp examples}
~
    \begin{enumerate}
        \item (\cite{yasuda2014cyclic}, Corollary 6.21) Let $k$ be an algebraically closed field of characteristic $p>0$. Suppose that $G\subseteq \SL(n,k)$ is a finite group with no pseudo-reflections, such that $G\cong C_p$, where $C_p$ is the $p$-cyclic group. If the associated quotient singularity has a crepant resolution $Y\to X:=\AAA^n_k/G$, then $e(Y)=\#\mathrm{Conj}(C_p)=p$.
        \item (\cite{yamamoto2021crepant}, Theorem 1.2) Let $k$ be an algebraically closed field of characteristic $3$. Suppose that $G\subseteq \SL(3,k)$ is a finite group with no pseudo-reflections, such that $G\cong H\rtimes C_3$, where $H$ is a non-modular abelian group. Then the associated quotient singularity has a crepant resolution $Y\to X:=\AAA^3_k/G$, and $e(Y)=\#\mathrm{Conj}(G)$.
        \item (\cite{fanmaster}, Corollary 3.1) Let $k$ be an algebraically closed field of characteristic $2$. For any positive odd number $l$, let $\zeta_l$ be a primitive $l$-th root of unity in $k$. For integers $a_1,a_2,a_3,a_4$, denote the diagonal matrix $\mathrm{diag}(\zeta_l^{a_1},\zeta_l^{a_2},\zeta_l^{a_3},\zeta_l^{a_4})$ by $\frac{1}{l}(a_1,a_2,a_3,a_4)$. Suppose that $G\subseteq \SL(4,k)$ has a semidirect product structure $G\cong H\rtimes C_2$, where $C_2$ is generated by the permutation $(12)(34)$ as an element of $\SL(4,k)$, and $H$ is one of the following:
        \begin{itemize}
            \item $H=\langle \frac{1}{l}(1,-1,0,0),\frac{1}{l}(1,0,-1,0),\frac{1}{l}(1,0,0,-1) \rangle$,
            \item $H=\langle \frac{1}{l_1}(1,-1,0,0),\frac{1}{l_2}(0,0,1,-1) \rangle$.
        \end{itemize}
        Then $G$ has no pseudo-reflections, and the associated quotient singularity has a crepant resolution $Y\to X:=\AAA^4_k/G$, satisfying $e(Y)=\#\mathrm{Conj}(G)$.
    \end{enumerate}
\end{thm}

In the results above, $G$ is always small (that is, $G$ has no pseudo-reflections) and has a semidirect product structure given by a non-modular abelian normal subgroup and a $p$-cyclic group. If $G\cong H\rtimes K$, then it is a classical idea to consider the decomposition $\AAA^n\xrightarrow{/H} \AAA^n/H \xrightarrow{/K} \AAA^n/G$ and the singularities given by $H$ and $K$ respectively. Under the assumption, $H$ is non-modular, and the singularity given by $K\cong C_p$ is studied by Theorem \ref{HCp examples}.1, with the same statement as the Batyrev's theorem. In fact, Theorem \ref{HCp examples}.2 and Theorem \ref{HCp examples}.3 are parallel to Ito's construction of crepant resolutions of trihedral singularities (\cite{ito1994crepant}) and specific $4$-dimensional quotient singularities (\cite{hayashi2017existence}) over $\CC$, with the same approach. In this paper, we show that for such small groups with semidirect product structures, the analogous statement of the Batyrev's theorem always holds.

\begin{thm}[Main theorem]\label{main}
    Let $k=\overline{\mathbb{F}_p}$ be the algebraic closure of the finite fields of characteristic $p>0$, and $G\subseteq \mathrm{SL}(n,k)$ be a finite small group, such that $G$ has a non-modular abelian normal subgroup $H$ of index $p$. 
    If $X:=\mathbb{A}^n_k/G$ has a crepant resolution $f:Y\to X$, then $$e(Y)=\#\mathrm{Conj}(G)=\#\mathrm{Ind}_k(G),$$
    where $\mathrm{Ind}_k(G)$ is the set of indecomposable $kG$-modules up to isomorphism.
\end{thm}

Thus, Theorem \ref{HCp examples}.2 and Theorem \ref{HCp examples}.3 become examples of the main theorem. By this result, we also point out that the number of isomorphism classes of indecomposable $kG$-modules, which is not necessarily equal to $\#\mathrm{Conj}(G)$ in positive characteristic, is a candidate for the algebraic invariant of $G$ corresponding to the Euler characteristic of the crepant resolution. Indeed, we have the following conjecture.

\begin{conj}
    Let $k$ be an algebraically closed field of characteristic $p>0$, and $G\subseteq \mathrm{SL}(n,k)$ be a finite small group of finite representation type (that is, there are finitely many isomorphism classes of indecomposable $kG$-modules). Let $P$ be a $p$-Sylow subgroup of $G$.
    If both $X:=\mathbb{A}^n_k/G$ and $X':=\mathbb{A}^n_k/P$ have crepant resolutions $f:Y\to X, f':Y'\to X'$, then $$e(Y)=\#\mathrm{Ind}_k(G).$$
\end{conj}

\begin{rem}
    The assumption of the existence of a crepant resolution of $\mathbb{A}^n_k/P$ is necessary here. Otherwise, $\mathbb{A}_k^6/S_3$ in characteristic $3$ gives a counterexample, where the symmetric group $S_3$ acts on $\mathbb{A}_k^6$ by the direct sum of two copies of the permutation representation. See Chapter 7.4 of \cite{brion2005frobenius} and Section 5.4 of \cite{wood2015mass} for more details.
\end{rem}
This conjecture is true for all known cases, including the non-modular case, the main theorem and Yamamoto's counterexamples in \cite{yamamoto2021crepant}. On the other hand, there are also examples where the group is of infinite representation type, such as the author's construction in \cite{fan2023crepant}. In the last section, two quotient singularities given by certain groups of infinite representation type are considered.

\begin{exm}[=Corollary \ref{examples infinite}]
    Let $k$ be an algebraically closed field of characteristic $2$, $A_4$ and $C_2^2$ be the subgroups of $\SL(4,k)$ given by the permutation actions of the alternating group and its $2$-Sylow subgroup respectively. Denote the associated quotient singularities by $X_1:=\AAA^4/{C_2^2}$ and $X_2:=\AAA^4/A_4$. If there exists a crepant resolution $Y_1\to X_1$ ($resp.$ $Y_2\to X_2$), then $e(Y_1)=6$ ($resp.$ $e(Y_2)=10$).
\end{exm}

This paper is organized as follows. In Section 2, we introduce the results and notations that are used for the proof of the main theorem, including the main tool - the mass formula version of the wild McKay correspondence. In Section 3, we prove the main theorem when $G$ is abelian. In Section 4, we prove the main theorem when $G$ is non-abelian. In Section 5, we provide two examples of computing the Euler characteristic of crepant resolutions for groups of infinite representation type, using the key idea of the proof of our main theorem.

\section{Preliminaries}

For the proof of our main theorem, the wild McKay correspondence as mass formulas is an important tool. In this version of the motivic wild McKay correspondence over finite fields, stringy motives are studied by their realization as the stringy point-counts $\#_\mathrm{st}$. Under some specific conditions, if there exists a crepant resolution, then some of their geometric properties can be studied via the stringy point-counts of the quotient varieties. For more details, the reader can refer to \cite{yasuda2017wild}. Here below we only list the results that are necessary for our proof.

\begin{defn}
    Let $K$ be a field and $G$ be a finite group. $M$ is called a $G$-\'etale $K$-algebra, if $M$ is a finite \'etale $K$-algebra of degree $\#G$ equipped with a $G$-action, such that $M^G=K$. Isomorphisms of $G$-\'etale $K$-algebras are the isomorphisms of $K$-algebras that are $G$-equivariant. $G\text{-}\mathrm{\Acute{E}t}(K)$ is the set consisting of the isomorphism classes of $G$-\'etale $K$-algebras.
\end{defn}

\begin{rem}
    Different actions of $G$ on the same $K$-algebra may result in different $G$-\'etale $K$-algebras. However, conjugate endomorphisms of $G$ do not change the class in $G\text{-}\mathrm{\Acute{E}t}(K)$. For each $M\in G\text{-}\mathrm{\Acute{E}t}(K)$, there exists a subgroup $G_M\subseteq G$ (that is unique up to conjugation) such that $M\cong L_M^{\oplus [G:G_M]}$, where $L_M$ is a Galois extension of $K$ with its Galois group isomorphic to $G_M$.
\end{rem}

\begin{thm}[\cite{yasuda2017wild}, Proposition 8.5]\label{wildmck}
    Let $k$ be a finite field, $\mathbb{F}_q/k$ be a field extension with order $q=p^e$, and $K=\mathbb{F}_q((t))$ be the local field of formal Laurent series. Let $G$ be a finite small subgroup of $\mathrm{SL}(n,k)$. Let $X:=\mathbb{A}^n/G$ be the associated quotient variety. Then 
    $$\#_{\mathrm{st},\mathbb{F}_q}X=\sum_{M\in G\text{-}\mathrm{\Acute{E}t}(K)}\frac{q^{n-\mathbf{v}_V(M)}}{\#C_G(G_M)};$$
    if furthermore $X$ has a crepant resolution $Y\to X$, then
    $$\#Y(\mathbb{F}_q)=\#_{\mathrm{st},\mathbb{F}_q}X=\sum_{M\in G\text{-}\mathrm{\Acute{E}t}(K)}\frac{q^{n-\mathbf{v}_V(M)}}{\#C_G(G_M)}.$$
    Here $\#Y(\mathbb{F}_q)$ counts the number of $\mathbb{F}_q$-rational points of $Y$, $C_G(G_M)$ is the centralizer of $G_M$ in $G$, and $\mathbf{v}_V(M)$ is the $\mathbf{v}$ function defined as follows.
\end{thm}

\begin{defn}[\cite{yasuda2017wild}; $\mathbf{v}$ function]
    Let $k$, $\mathbb{F}_q$ and $K$ be the fields of characteristic $p$ as stated in Theorem \ref{wildmck}. Let $\mathcal{O}_K:=\mathbb{F}_q[[t]]$ be the ring of integers. $\mathbf{v}$ function is defined for an arbitrary finite group $G$ with a representation $\rho : G\to \mathrm{GL}(n,k)\hookrightarrow \mathrm{GL}(n,\mathcal{O}_K)$. Let $V=\mathbb{A}^n_{\mathcal{O}_K}$ be the representation space and $\mathcal{O}_K[x_1,\dots,x_n]$ be its coordinate ring with the linear part $T=\sum \mathcal{O}_K x_i$. Then for any $M\in G\text{-}\mathrm{\Acute{E}t}(K)$, $G$ has both an action on $\mathcal{O}_M$ via the definition of $G$-\'{e}tale $K$-algebras and an action on $T$ via the representation $\rho$. Denote by $\mathrm{Hom}_{\mathcal{O}_K}^G(T,\mathcal{O}_M)$ the homomorphisms that commute with the actions of $G$. Then $\mathrm{Hom}_{\mathcal{O}_K}^G(T,\mathcal{O}_M)$ is a free $\mathcal{O}_K$-module of rank $n$ (\cite{yasuda2017toward}, Proposition 6.3), and
    $$\mathbf{v}_V(M):=\frac{1}{\#G}\mathrm{length}_{\mathcal{O}_   K}\frac{\mathrm{Hom}_{\mathcal{O}_K}(T,\mathcal{O}_M)}{\mathcal{O}_M\mathrm{Hom}_{\mathcal{O}_K}^G(T,\mathcal{O}_M)}.$$
\end{defn}

\begin{rem}
    $\mathbf{v}$ function is convertible and additive (\cite{wood2015mass}, Lemma 3.4). In this paper, we will use the following facts: Even if $\rho$ is not faithful, it can be decomposed as $G\rightarrow G/H \rightarrow \mathrm{GL}(n,k)$ for some $H\subseteq G$, and $\mathbf{v}_V(M)=\mathbf{v}_{V_{G/H}}(M^H)$; if $V$ has a direct decomposition as $\displaystyle V=\bigoplus_{i=1}^s V_i$, then $\displaystyle\mathbf{v}_V=\sum_{i=1}^s \mathbf{v}_{V_i}$.
\end{rem}
\begin{rem}
    The value of the $\mathbf{v}$ function for $M\in G\text{-}\mathrm{\Acute{E}t}(K)$ is determined by $L_M/K$ as a $G_M$-\'{e}tale $K$-algebra. We simply write $\mathbf{v}(L)$ if $G_M$ is fixed.
\end{rem}

Next, we give some notations and use them to rewrite the wild McKay correspondence such that the Euler characteristic of crepant resolutions can be studied.
\begin{defn}
    Denote $$A(G):=\{L\text{: a }G\text{-\'etale }K\text{-algebra}\mid L\text{ is a Galois extension of }K\}.$$ We define two series of $q$, denoted by $f_{G}$ and $F_{G}$, as follows.
    $$f_{G}(q):=\sum_{L\in A(G)}q^{n-\mathbf{v}(L)},$$
    $$F_{G}(q):=\sum_{[G']:G'\subseteq G}\frac{1}{\#N_G(G')}f_{G'}.$$
    Here $N_G(G')$ is the normalizer of $G'$ in $G$, and $[G']$ runs over the classes of subgroups of $G$ up to conjugation.
\end{defn}

\begin{rem}
    Every element in $A(G)$ contains the information of a Galois extension $L/K$ with Galois group $\mathrm{Gal}(L/K)\cong G$ and the action of $G$ on it. For example, if $G\cong C_l$ for some prime $l$ ($l\neq p$) and the order of $k$ is large enough, then by Kummer theory, there are $l+1$ Galois extensions of $K=k((t))$ with $l$-cyclic Galois group up to isomorphism: one is unramified and the other $l$ extensions are totally ramified. On the other hand, for each extension, there are $l-1$ ways to map the given generator of $G$ to a nontrivial element in the Galois group, giving different actions up to isomorphism of $G$-\'{e}tale algebras. Hence $A(C_l)$ has $l^2-1$ elements.
\end{rem}

\begin{rem}
    Consider the conjugation action by $h\in N_G(G_M)$ on $G_M$. For any $L\in A(G_M)$, the conjugation by $h$ may change the action of $G_M$ on the \'{e}tale $K$-algebra. Therefore, $[N_G(G_M):C_G(G_M)]$ distinct elements in $A(G_M)$ are in the same isomorphism class in $G\text{-\'Et}(K)$. From this observation, $F_G$ is exactly the right-hand side of Theorem \ref{wildmck}.
\end{rem}

\begin{defn}
    Let $f$ be a Laurent polynomial with coefficients in $\mathbb{Q}$. We define $S(f):=f(1)$, which is equal to the sum of the coefficients of $f$. Therefore, $S$ is a homomorphism of $\mathbb{Q}$-algebras.
\end{defn}

\begin{cor}\label{euler}
    Under the assumption of Theorem \ref{wildmck}, if there exists a crepant resolution $Y$, then $F_G(T^2)$ is a $\mathbb{Z}$-coefficient polynomial in $T$, and $e(Y)=S(F_G)$. 
\end{cor}

\begin{proof}
    The first assertion is from the motivic wild McKay correspondence (construction of stringy motives and their realization as Poincar\'e polynomials of crepant resolutions via $\mathbb{L}\mapsto T^2$), and the second one follows; one can also check it by the functional equation in the Weil conjectures. The reader can refer to \cite{yasuda2014cyclic} for details.
\end{proof}

\begin{prop}\label{nonmodmck}
    Assume that the order of $k$ is large enough. If $G$ is non-modular, then $F_G$ is a polynomial in $q$ and $S(F_G)=\#\mathrm{Conj}(G)$. If $G$ is also abelian, then $\displaystyle\sum_{G'\subseteq G}f_{G'}=(\#G)^2$. 
\end{prop}
\begin{proof}
    This is a corollary of the tame McKay correspondence (\cite{wood2015mass}, Section 5.2). We also point out that $S(f_G)=\#A(G)$ here.
\end{proof}
\begin{prop}\label{pcyclicmck}
    If $G$ is isomorphic to a cyclic group of order $p$ and there exists a crepant resolution $Y\to X$ , then $F_G$ is a polynomial in $q$, such that $S(F_G)=p$ and $S(f_G)=p^2-1$.
\end{prop}
\begin{proof}
    This is a corollary of the $p$-cyclic McKay correspondence (\cite{yasuda2014cyclic}, Corollary 6.21).
\end{proof}

Then, we give the proof of the equality between the number of conjugacy classes and the number of isomorphism classes of indecomposable representations in the statement of our main theorem.
\begin{prop}\label{conj and ind}
    Let $G$ be a finite group equipped with a semidirect product structure $H\rtimes C_p$, where $H$ is abelian and $C_{p}$ is a cyclic group of
prime order $p$, such that $p \nmid \# H$. Let $C(G)$ be the center of $G$ and $k$ be a splitting field of $G$ of characteristic $p$ (equivalently, for any element $h\in H$, $k$ contains all of the $\mathrm{ord}(h)$-th primitive roots of unity in $\overline{k}$). Then
    $$\#\mathrm{Conj}(G)=\#\mathrm{Ind}_k(G)=\left\{
\begin{array}{ll}\displaystyle
p\#C(G)+\frac{\#H-\#C(G)}{p} & G\text{ is not abelian,} \\
\#G & G\text{ is abelian.}
\end{array}
\right.$$
\end{prop}
\begin{proof}
    We first consider the case when $G$ is not abelian. Then $G$ contains $\#C(G)$ conjugacy classes of one element, $(\#H-\#C(G))/{p}$ conjugacy classes of $p$ elements and $(p-1)\#C(G)$ conjugacy classes of $[H:C(G)]$ elements. Hence the number of conjugacy classes of $G$ is equal to the right-hand side of the proposition.

    On the other hand, by a corollary of Green correspondence (one can refer to \cite{webb2016course}, Corollary 11.6.5), for $G$ with its $p$-Sylow subgroup isomorphic to $C_p$, if one denotes the number of simple $kG$-modules up to isomorphism by $l_k(G)$, then
    $$\#\mathrm{Ind}_k(G)=l_k(G)+(p-1)l_k(N_G(C_p)).$$
    By Brauer's theorem on $l_k(G)$, we have $\displaystyle l_k(G)=\#C(G)+(\#H-\#C(G))/{p}$ and $l_k(N_G(C_p))=l_k(C(G)\times C_p)=\#C(G)$. Therefore, the number of isomorphism classes of indecomposable $kG$-modules coincides with $\#\mathrm{Conj}(G)$.

    If $G$ is abelian, then $\#\mathrm{Conj}(G)=\#G=p\#H$, $l_k(G)=\#H$ and $N_G(C_p)=G$. Then an easy computation shows the statement again.
\end{proof}

In the last part of this section, we introduce the main idea for the proof of our main result. By Corollary \ref{euler} and Proposition \ref{conj and ind}, to prove the main theorem, it suffices to show that under the assumption of the main theorem, if there exists a crepant resolution, then $$\displaystyle S(F_G)=\#\mathrm{Conj}(G).$$ Note that even though we use a version of the wild McKay correspondence over finite fields, we can consider the main theorem over some appropriate finite fields instead, such that the $G$-actions, quotient singularities and resolutions are obtained by base change. In the remaining parts, we always assume that the order of the finite field $k=\mathbb{F}_q$ is large enough, such that it is a representation-theoretic splitting field of $G$ in characteristic $p$; in other words, not only the Euler characteristic, but also the representation-theoretic invariants in the main theorem remain unchanged under the base change to $\overline{k}$.

\section{The abelian case of the main theorem}

The proof starts from the case when $G$ is abelian. Now let $G\cong H\times C_p$ for an abelian non-modular subgroup $H$. We first give an observation on the structure of such $G$.

\begin{lem}\label{structure of abelian HCp}
    Let $G\cong H\times C_p$ be a finite subgroup of $\mathrm{GL}(n,k)$. Then the representation space $V$ has a direct sum decomposition $\displaystyle V\cong\bigoplus_{i=1}^s V_i$, such that each $V_i$ induces an indecomposable $C_p$-representation of dimension $1\leqslant d_i \leqslant p-1$ and $\forall h\in H$, $h$ acts on $V_i$ as a scalar multiplication.
\end{lem}

\begin{proof}
    For $G\cong H\times C_p$, we can assume that $H$ consists of diagonal matrices by similarity transformation. Denote the matrix of a generator of $C_p$ under such transformation by $a=(a_{ij})_{i,j}$. $\forall h=\mathrm{diag}(\alpha_1,\dots,\alpha_n)\in H$, from $ha=ah$ we have $\alpha_i a_{ij}=\alpha_j a_{ij}$ for any $i,j$. If any $h$ in $H$ is of the form $\mathrm{diag}(c,\dots,c)$, then the lemma is trivial. So we can assume that there exists some $\alpha_i \neq \alpha_j$, and therefore $a_{ij}=a_{ji}=0$. By changing the order of the basis if necessary, this implies that $a$ is a block diagonal matrix. In other words, $V$, as a representation space of $G$, can be decomposed into a direct sum of subrepresentations as eigenspaces of different eigenvalues of $h$. Repeat this process on each subspace in the obtained decomposition for all $h\in H$, and then we obtain $\displaystyle V=\bigoplus_t W_t$, such that $\forall h\in H$ and $\forall W_t$, $W_t$ is a subspace of an eigenspace of $h$. Decompose each $W_t$ into the direct sum of indecomposable representations of $C_p$. Then every summand is still a $G$-representation, and the desired decomposition $\displaystyle V=\bigoplus_{i=1}^s V_i$ is obtained.
\end{proof}

\begin{rem}
    Consequently, we can choose a basis of each $V_i$ appropriately, such that $H$ consists of diagonal matrices, and a generator of $C_p$ has the form $\mathrm{diag}(J_{d_1}(1), J_{d_2}(1),\dots,J_{d_s}(1))$, where $J_i(\alpha)$ denotes the Jordan block of size $i$ and eigenvalue $\alpha$.
\end{rem}

Then we want to compute $f_{G}$ for $G\cong H\times C_p$, a small finite subgroup of $\SL(n,k)$. 

For an element $L\in A(G)$, forgetting the $G$-action on it, $L$ is a Galois extension obtained by choosing a Galois extension $L_1/K$ with Galois group $\mathrm{Gal}(L_1/K)\cong H$ and an Artin-Schreier extension $L_2/K$, such that $L/K=L_1L_2/K$. For $L_2$, it depends on the choice of an element $a\in K$ such that $L_2=K(\alpha)$ for some $\alpha$ satisfying $\wp(\alpha)=\alpha^p-\alpha=a$. We may choose $\alpha$ such that the fixed generator of $C_p$ maps $\alpha$ to $\alpha+1$. 
According to \cite{yasuda2014cyclic}, Proposition 2.5, for this $C_p$-\'etale $K$-algebra, up to isomorphism, the element $a$ for the Artin-Schreier extension can be chosen from
$$K/\wp(K)\setminus\{0\}\cong \mathbb{F}_{q}/\wp(\mathbb{F}_{q})\oplus \bigoplus_{j>0,(p,j)=1}\mathbb{F}_{q}t^{-j}\setminus\{0\}.$$

In other words, every element $L\in A(G)$ is determined by an element $L_1\in A(H)$ and an element $L_2 \in A(C_p)$ with $a\in K/\wp(K)\setminus\{0\}$. If we fix $L_1$ and the degree $-j\in \{n\in\mathbb{Z}\mid n\leqslant 0, n\notin p\mathbb{Z}_{-}\}$, then there are
$$\left\{
\begin{array}{ll}
p-1 & j=0 \\
pq^{j-1-\floor*{\frac{j-1}{p}}}(q-1) & j>0
\end{array}
\right.$$
choices of the $C_p$-\'etale extension $L_2$ generated by the element $\alpha$ with $v_{L_2}(\alpha)=-j$. $L_2$ is unramified if and only if $j=0$. 

Fix $L\in A(G)$ and continue using the notations $L_1, L_2, j$. We want to compute the value of $\mathbf{v}(L)$. 

\begin{lem}\label{ignore unram}
    $\mathbf{v}(L)=\mathbf{v}(L/L_1^\mathrm{ur})$. Here $L_1^\mathrm{ur}$ is the maximal unramified extension of $L_1$, $\mathbf{v}(L/L_1^\mathrm{ur})$ is the value of the $\mathbf{v}$ function defined for the group $\overline{G}:=G/\mathrm{Gal}(L_1^\mathrm{ur}/K)$ over $L_1^\mathrm{ur}\cong \mathbb{F}_{q^f}((t))$, and $f$ is the residue degree of $L_1/K$.
\end{lem}
\begin{proof}
    Denote the dual basis of $T=\sum \mathcal{O}_K x_i$ in $\mathrm{Hom}_{\mathcal{O}_K}(T,\mathcal{O}_L)$ by $x_1^*,\dots,x_n^*$. Then by \cite{yasuda2017toward}, Proposition 6.3, $\mathrm{Hom}_{\mathcal{O}_K}^{G}(T,\mathcal{O}_L)$ is a free $\mathcal{O}_K$-module of rank $n$. If $(\sum c_{ij}x_j^*)_{1\leqslant i\leqslant n}$ is an $\mathcal{O}_K$-basis of $\mathrm{Hom}_{\mathcal{O}_K}^{G}(T,\mathcal{O}_L)$, then $$\mathbf{v}(L)=\frac{1}{\#G}\mathrm{length}_{\mathcal{O}_K}\frac{\mathcal{O}_L}{(\mathrm{det}(c_{ij}))}=\frac{v_{L}(\mathrm{det}(c_{ij}))f_{L/K}}{e_{L/K}f_{L/K}}=\frac{v_{L}(\mathrm{det}(c_{ij}))}{e_{L/K}}.$$

    By abuse of notation, let $(\sum a_{ij}x_j^*)_{1\leqslant i\leqslant n}$ be an $\mathcal{O}_{L_1^\mathrm{ur}}$-basis of the module $\mathrm{Hom}_{\mathcal{O}_{L_1^\mathrm{ur}}}^{\overline{G}}(T,\mathcal{O}_L)$. On the other hand, there exists an element $g\in \mathrm{Gal}(L_1^\mathrm{ur}/K)$ of order $f$ and $c\in \mathbb{F}_{q^f}\hookrightarrow L_1^\mathrm{ur}$ such that 
    $$g(c)=\zeta_f c.$$
    The corresponding matrix of $g$ in $G$ has the form $\mathrm{diag}(\zeta_f^{a_1},\dots,\zeta_f^{a_n})$ ($0\leqslant a_i <f$ for each $i$). Then $(c^{a_1}x_1^*,\dots,c^{a_n}x_n^*)$ forms an $\mathcal{O}_{K}$-basis of the module $\mathrm{Hom}_{\mathcal{O}_{K}}^{\mathrm{Gal}(L_1^\mathrm{ur}/K)}(T,\mathcal{O}_{L_1^\mathrm{ur}})$. What is more, $(\sum a_{ij}c^{a_j} x_j^*)_{1\leqslant i\leqslant n}$
    is an $\mathcal{O}_K$-basis of $\mathrm{Hom}_{\mathcal{O}_K}^{\overline{G}}(T,\mathcal{O}_L)$. Note that $v_L(c)=0$, and we have
    $$\mathbf{v}(L)=\frac{v_{L}(\mathrm{det}(c^{a_j} a_{ij}))}{e_{L/K}}=\frac{v_{L}(\mathrm{det}(a_{ij}))}{e_{L/L_1^\mathrm{ur}}}=\mathbf{v}(L/L_1^\mathrm{ur}).$$

\end{proof}

By Lemma \ref{structure of abelian HCp} and Lemma \ref{ignore unram}, together with the convertibility and additivity of $\mathbf{v}$ functions, we only need to consider the case when $L_1$ is a cyclic totally ramified extension and the representation is indecomposable.

\begin{lem}
    Assume that $1\leqslant n\leqslant p$, and that $G$ is generated by $\sigma,\tau$ of the form
    $$\sigma=J_n(1), \tau =\mathrm{diag}(\zeta_l,\dots,\zeta_l),$$ and that $L_1/K$ is totally ramified with ramification index $l$. (We allow $G\subseteq \mathrm{GL}(n,k)$ here.) 
    
    Then when $l>1$,
    $$
    \mathbf{v}(L)=
\frac{n}{l}+\sum_{i=1}^n\ceil*{\frac{(i-1)j}{p}-\frac{1}{l}} 
.
$$
When $l=1$, 
$$\mathbf{v}(L)=\sum_{i=1}^n\ceil*{\frac{(i-1)j}{p}}.$$
\end{lem}

\begin{proof}

    For the case when $l=1$, the value of the $\mathbf{v}$ function has been computed in \cite{yasuda2017toward}, Example 6.8. 
    
    Now we assume that $n,l>1$. 
    Without loss of generality, we can choose the generator $\alpha$ of $L_2$ and a uniformizer $\beta$ of $\mathcal{O}_{L_1}$ such that $$\sigma(\alpha)=\alpha+1, \sigma(\beta)=\beta,\tau(\alpha)=\alpha,\tau(\beta)=\zeta_l \beta.$$

    We define the formal binomial coefficients 
    $$\binom{\alpha}{i}:=\frac{\alpha(\alpha-1)\dots(\alpha-i+1)}{i!}$$
    for $i=1,\dots,n-1$. Since $i<n\leqslant p$, they are well-defined. Additionally denote $\binom{\alpha}{0}=1$. Then the formal binomial coefficients satisfy $$\binom{\alpha+1}{i}=\binom{\alpha}{i}+\binom{\alpha}{i-1}.$$ Let $x_1^*,\dots,x_n^*$ be the dual basis of $T=\sum \mathcal{O}_K x_i$ in $\mathrm{Hom}_{\mathcal{O}_K}(T,\mathcal{O}_L)\hookrightarrow \mathrm{Hom}_K(\sum Kx_i,L)$. Then the actions of $\sigma$ and $\tau$ on $T$ are given by
    $$\sigma(x_i)=\left\{
\begin{array}{ll}
x_1 & i=1 \\
x_{i-1}+x_i & i\neq 1
\end{array}
\right.
, \tau(x_i)=\zeta_l x_i.$$

Then $\mathrm{Hom}^G_K(\sum Kx_i,L)$ is generated by 
$$
\begin{pmatrix}
\beta x_1^* & \cdots & \beta x_n^*
\end{pmatrix}
\begin{pmatrix}
  0      & 0                  & \cdots  &  0  &        1              \\
  0      & 0                  & \cdots  &  1  &  \binom{\alpha}{1}   \\   
  \vdots & \vdots             & \iddots   &\vdots &\vdots           \\
  0      & 1                  & \cdots & \cdots & \binom{\alpha}{n-2}  \\
  1      & \binom{\alpha}{1} & \cdots & \binom{\alpha}{n-2} & \binom{\alpha}{n-1}  \\
\end{pmatrix}
$$
over $K$. Denote this basis by $(\phi_1,\dots,\phi_n)=\beta(x_1^*,\dots,x_n^*)\mathbf{A}$, and then the matrix $\mathbf{A}$ has determinant $\pm 1$. If we take $$m_i:=\ceil*{\frac{-(i-1)v_L(\alpha)-v_L(\beta)}{e_{L/K}}}=\ceil*{\frac{(i-1)j}{p}-\frac{1}{l}},$$
then $\{t^{m_1}\phi_1,\dots,t^{m_n}\phi_n\}$ forms an $\mathcal{O}_K$-basis of $\mathrm{Hom}_{\mathcal{O}_K}^{G}(T,\mathcal{O}_L)$, and 
$$\mathbf{v}(L)=\frac{1}{e_{L/K}}v_{L}(\beta^n t^{m_1+\dots+m_n}\mathrm{det}\mathbf{A})$$
shows the lemma.

For the case when $n=1$, it is easy to show the lemma with a similar statement only considering the tame part.
    
\end{proof}

Now we go back to the general case for the abelian group $G\cong H\times C_p$ and the $G$-\'{e}tale $K$-algebra $L\in A(G)$.

\begin{defn}
    Consider the decomposition $V=\bigoplus_{i=1}^s V_i$ as in Lemma \ref{structure of abelian HCp}. Every indecomposable summand $V_i$ induces a Galois extension $N_i/L_1^\mathrm{nr}$, of which the tame ramification index is $l_i\geqslant 1$ and the Galois group is isomorphic to $C_{l_i}\times C_p$. We define the generalized upper shift number $$\sht_V^{L}(j):=\sum_{i=1}^s\sum_{k=1}^{d_i}\ceil*{\frac{(k-1)j}{p}-\frac{1}{l_i}+\floor*{\frac{1}{l_i}}}$$
    and the generalized age
    $$\mathrm{age}(L_1):=\sum_{i=1}^s d_i\left(\frac{1}{l_i}-\floor*{\frac{1}{l_i}}\right).$$
\end{defn}

\begin{rem}
    If $H=\{e\}$, then $\sht ^L_V(j)=\mathrm{sht}_V(j)+n-s$, where $\mathrm{sht}_V$ is the shift number defined in \cite{yasuda2014cyclic}. If $d_i=1$ for any $i$, then $\mathrm{age}(L_1)=\mathrm{age}(g)$, where $g\in G$ satisfies that for any $N_i=L_1^\mathrm{nr}(t_i)$, $g(t_i)=\zeta_{l_i}t_i,$ and $\mathrm{age}(g)$ is the age grading defined in \cite{ito1996mckay}.
\end{rem}

\begin{cor}
    $\displaystyle\mathbf{v}(L)=\mathrm{age}(L_1)+ \sht^L_V(j)
.$ 
\end{cor}

Note that $l_i$ and $\mathrm{age}(L_1)$ are completely determined by $L_1$.
For the part of generalized upper shift numbers, it is determined by $L_1$ and $j$. We also have the following property by an easy computation.
\begin{prop}
Denote $D_V:=\displaystyle \sum_{i=1}^s \frac{d_i(d_i-1)}{2}$. Then for each $j=ap+r$ ($a\geqslant 0, 0< r<p$),
$$\sht^L_V(j)=aD_V+\sht^L_V(r).$$
\end{prop}

With this property of $\sht^L_V(j)$, we can compute $f_G$.
    
\begin{prop}\label{lp wild mckay}
    Let $G\cong H\times C_p$. If there exists a crepant resolution $Y\to X$, then $D_V=p$, and every $f_{G'}$ with $G'\cong H'\times C_p\subseteq G$ is a polynomial in $q$, with $S(f_{G'})=(p^2-1)S(f_{H'}).$
\end{prop}

\begin{proof}

    \begin{align*}
        f_G(q) &= \sum_{L\in A(G)} q^{n-\mathbf{v}(L)} \\
               &= \sum_{L_1 \in A(H)} ((p-1)q^{n-\mathrm{age}(L_1)}\\
               &+\sum_{j>0,(p,j)=1}pq^{j-1-\floor*{\frac{j-1}{p}}}(q-1)q^{n-(\mathrm{age}(L_1)+ \sht^L_V(j))}) \\
               &= \sum_{L_1\in A(H)}q^{n-\mathrm{age}(L_1)}(p-1 \\
               &+ p(1-q^{-1})\sum_{a=0}^\infty \sum_{r=1}^{p-1}q^{ap+r-a-(aD_V+\sht^L_V(r))})\\
               &= \sum_{L_1\in A(H)}q^{n-\mathrm{age}(L_1)}(p-1 \\
               &+ p(1-q^{-1})\sum_{r=1}^{p-1} q^{r-\sht^L_V(r)}\sum_{a=0}^\infty q^{a(p-1-D_V)}).
    \end{align*}

    Therefore, $f_G(q)$ is a rational function if and only if $D_V\geqslant p$; if so, we can write
    \begin{align*}
      &f_G(q)\\
      = &\sum_{L_1\in A(H)}q^{n-\mathrm{age}(L_1)}\left(p-1 
    + p\frac{1-q^{-1}}{1-q^{p-1-D_V}}\sum_{r=1}^{p-1} q^{r-\sht^L_V(r)}\right).  
    \end{align*}

    If this rational function is furthermore a polynomial, then $D_V=p$, and the value of $f_G$ at $q=1$ is $S(f_G)=f_G(1)=\#A(H)(p-1+p(p-1))=(p^2-1)S(f_H)$.

    Since $D_V$ is equal for all $G'\cong H'\times C_p \subseteq G$, we complete the proof of this proposition by similar computations.
    
\end{proof}

\begin{rem}
    As a byproduct of this proposition, if $D_V\geqslant p$, then the quotient singularity $X$ is log terminal because of the convergence of its stringy motif. The reader can refer to \cite{yasuda2014cyclic} for details in the $p$-cyclic McKay correspondence.
\end{rem}

\begin{cor}\label{abelian HCp}
    Our main theorem holds when $G$ is abelian.
\end{cor}

\begin{proof}
    By Proposition \ref{nonmodmck} and Proposition \ref{lp wild mckay}, if there exists a crepant resolution, then
    \begin{align*}
        S(F_G)&=\frac{\sum_{G'\subseteq G}S(f_G)}{\#G}=\frac{1}{\#G}\sum_{H'\subseteq H}(S(f_{H'})+S(f_{H'\times C_p}))  \\
        &=\frac{p^2(\#H)^2}{\#G}=\#G.
    \end{align*}
\end{proof}

\section{The non-abelian case of the main theorem}

To study the non-abelian case, we need the following lemma.

\begin{lem}\label{notHCp}
    Let $G$ be a non-abelian group with a semidirect product structure $H\rtimes C_p$ as stated in Proposition \ref{conj and ind}. Then there does not exist any Galois extension $L/K$ such that $\mathrm{Gal}(L/K)\cong G$. In particular, $f_G=0$.
\end{lem}
\begin{proof}
    If there exists such a Galois extension $L/K$, then $L^{C(G)}/K$ is a Galois extension with its Galois group isomorphic to $G/C(G)\cong (H/C(G))\rtimes C_p$. Therefore, we can assume $C(G)=\{e\}$ without loss of generality.

    Assume that $G$ is the Galois group of a Galois extension $L/K$. Consider the ramification groups$$G=G_{-1}\supseteq G_0\supseteq G_1\supseteq \dots.$$
    By classical results (one can refer to \cite{serre2013local}, IV, Corollary 4), $G_0$, as a normal subgroup of $G$, is a semidirect product of a normal $p$-subgroup and a non-modular cyclic group. If $G_0$ contains an element of order $p$, then $G_0$ has to be the whole $G$ since $G_0\triangleleft G$, which contradicts the semidirect product structure of $G_0$. Hence $C_p$ is not contained in $G_0$. Since $G_0=\{g\in G\mid g\text{ acts trivially on }\mathrm{res}(L)\}$, $C_p$ acts effectively on $\mathrm{res}(L)$.

    Thus, for $L^H/K$, as an Artin-Schreier extension that is either unramified or totally ramified, its residue degree $f_{L^H/K}>1$. Consequently, $L^H/K$ is unramified. Therefore, if we write the local field $L\cong \mathrm{res}(L)((t_L))$, there exists an element $g\in G$ of order $p$ acting trivially on $t_L$. We assume that $C_p$ is generated by such $g$ without loss of generality. 

    Take an element $\sigma$ of prime order $l$ ($l\neq p$) from $H$. Then $L/L^{<\sigma>}$ is studied by Kummer theory. If $L/L^{<\sigma>}$ is unramified, then we can similarly assume that $\sigma$ acts trivially on $t_L$, and thus $\sigma$ has to commute with $C_p$, since the Galois groups of extensions of finite fields are abelian. If $L/L^{<\sigma>}$ is totally ramified, then $\sigma$ acts trivially on $\mathrm{res}(L)$, and hence also commutes with $C_p$.

    Either way, we can obtain a nontrivial element in $H$ that commutes with $C_p$ and thus lies in $C(G)$. Then the lemma is shown by contradiction, considering the assumption that $C(G)=\{e\}$.
\end{proof}

\begin{rem}
    Lemma \ref{notHCp} only holds in equal characteristic. In mixed characteristic, for example, the splitting field of $x^{13}+3$ over $\mathbb{Q}_3$ has a non-abelian Galois group isomorphic to $C_{13} \rtimes C_3$ (\cite{lmfdb:3.13.12.1}).
\end{rem}

\begin{proof}[Proof of the main theorem]
Since the abelian case has been shown in Corollary \ref{abelian HCp}, we only need to prove the main theorem for a non-abelian group $G\cong H\rtimes C_p$ now. In the non-abelian case, $C(G)\subseteq H$ and $N_G(C_p)=C(G)\times C_p$.

For the polynomial $F_G$, we first write
    $$S(F_G)=S\left(\sum_{[G']:G'\subseteq H}\frac{f_{G'}}{\#N_G(G')}+\sum_{[G']:\text{modular}}\frac{f_{G'}}{\#N_G(G')}\right).$$
Note that for $G'\subseteq H$, if $N_G(G')=G$, then the class $[G']$ contains one element; if $N_G(G')\neq G$, then $N_G(G')= H$ and the class $[G']$ contains $p$ elements. Either way, if we write the sum over $G'\subseteq H$, then the denominator is $\#G=p\#H$. By Proposition \ref{nonmodmck},
    $$S\left(\sum_{[G']:G'\subseteq H}\frac{f_{G'}}{\#N_G(G')}\right)=S\left(\sum_{G'\subseteq H}\frac{f_{G'}}{\#G}\right)=\frac{(\#H)^2}{\#G}=\frac{\#H}{p}.$$
    By Lemma \ref{notHCp}, for modular subgroups $G'\subseteq G$, $f_{G'}=0$ unless $G'\subseteq C(G)\times C_p$. Therefore, by Proposition \ref{lp wild mckay}, assuming the existence of crepant resolutions,
    \begin{align*}
        &S\left(\sum_{[G']:\text{modular}}\frac{f_{G'}}{\#N_G(G')}\right)=S\left(\sum_{C_p\subseteq G'\subseteq C(G)\times C_p}\frac{f_{G'}}{p\#C(G)}\right)\\
        &=\sum_{H'\subseteq C(G)}\frac{(p^2-1)S(f_{H'})}{p\#C(G)}=\left(p-\frac{1}{p}\right)\#C(G).
    \end{align*}

    Hence $$S(F_G)=p\#C(G)+\frac{\#H-\#C(G)}{p}=\#\mathrm{Conj}(G).$$
\end{proof}

\section{Computation of the Euler characteristic}

The idea of the proof of our main theorem can also be applied to compute the Euler characteristic in other cases: in \cite{yamamoto2022mass}, Yamamoto computed a series of Euler characteristics, when the group has a semidirect product structure of a non-modular abelian normal subgroup and the symmetric group $S_3$, in characteristic $3$ and dimension $3$. In this section, we introduce two examples where the $p$-Sylow subgroup of $G$ is not cyclic (equivalently, the group $G$ is of infinite representation type in characteristic $p$; this is a result by Higman\cite{higman1954indecomposable}), and see that the Euler characteristic can still be computed via the wild McKay correspondence. To compute the $\mathbf{v}$ function in the examples, we need the theorem below.

\begin{thm}[\cite{wood2015mass}, Theorem 4.8]\label{artin conductor}
    Let $G$ be a finite group acting on a representation space $V$. The fields $k,K,L$ and the $\mathbf{v}$ function are as defined in Section 2. If $G$ acts on $V$ by permutation, then $\displaystyle\mathbf{v}(L)=\frac{1}{2}\mathbf{a}(L)$, where $$\mathbf{a}(L):=\sum_{i=0}^\infty \frac{\mathrm{codim}(k^n)^{G_i}}{[G_0:G_i]}$$ is the Artin conductor of this permutation representation.
\end{thm}

In the remaining part of this section, $k$ is a finite field of order $q>2$ and characteristic $2$, $A_4$ is the alternating group with the permutation action on $\AAA_k^4$, and $C_2^2=\{e,(12)(34),(13)(24),(14)(23)\}$ is the normal subgroup containing modular elements in $A_4$, also acting on $\AAA^4$ by permutation.

\begin{prop}\label{v funtion for C22}
    For an element $L\in A(C_2^2)$, the corresponding extension is determined by the composition of two Artin-Schreier extensions, which are given by two different elements
    $$a,b\in k/\wp(k)\oplus \bigoplus_{j>0,(2,j)=1}kt^{-j}\setminus\{0\}.$$
    Denote $v_K(a)=-j$ and $v_K(b)=-k$, and assume that $j\leqslant k$. Then
    $$\mathbf{v}(L)=\left\{
\begin{array}{ll}
\displaystyle \frac{1}{2}(j+1)+k+1 & j>0 \\
\displaystyle k+1 & j=0
\end{array}
\right..$$
\end{prop}

\begin{proof}
    By \cite{wu2010ramification}, Theorem 3.11, if $L$ is totally ramified, then the ramification groups of $C_2^2$ are
    $$G_0=\dots=G_j=C_2^2, G_{j+1}=\dots=G_{j+2(k-j)}=C_2,G_{j+2(k-j)+1}=\{e\}.$$
    Therefore, by Theorem \ref{artin conductor}, when $j>0$, $\mathbf{v}(L)=\frac{3}{2}(j+1)+\frac{1}{2}(2k-2j)=\frac{1}{2}(j+1)+k+1$; when $j=0$, the ramification groups are determined only by the Artin-Schreier extension corresponding to $b$, and $\mathbf{v}(L)=k+1$.
\end{proof}

\begin{prop}\label{v function for A4}
    For an element $L\in A(A_4)$, the corresponding extension is determined by $L_1=K(\gamma)\in A(C_3)$ and an Artin-Schreier extension over $L_1$ given by an element
    $$a\in \left\{
\begin{array}{ll}
\displaystyle \bigoplus_{j>0,(2,j)=1}(\gamma k\oplus \gamma^2k)t^{-j} & L_1/K\text{ is unramified} \\
\displaystyle \bigoplus_{j>0,(6,j)=1}k\gamma^{-j} & L_1/K\text{ is ramified}
\end{array}
\right..$$
    Denote $v_{L_1}(a)=-j$, and then
    $$\mathbf{v}(L)=\left\{
\begin{array}{ll}
\displaystyle \frac{3}{2}(j+1) & L_1/K\text{ is unramified} \\
\displaystyle \frac{1}{2}(j+3) & L_1/K\text{ is ramified}
\end{array}
\right..$$
\end{prop}
\begin{proof}
    Assume that $L=L_1(\alpha,\beta)$, where $L_1(\alpha)$ and $L_1(\beta)$ are Artin-Schreier extensions corresponding to $a,b\in L_1$. All the nontrivial intermediate fields of $L/L_1$ are $L_1(\alpha),L_1(\beta)$ and $L_1(\alpha+\beta)$. Then any element of order $3$ in $\mathrm{Gal}(L/K)\cong A_4$ should give a cyclic permutation on $\alpha,\beta$ and $\alpha+\beta$. 

    In other words, $L/L_1$ is determined by an element $a\in L_1$ such that for a generator $\tau\in \mathrm{Gal}(L_1/K)$, $a,\tau(a)$ and $\tau^2(a)$ are distinct and $\tau^2(a)=\tau(a)+a$. Assume that $L_1=K(\gamma)$ and $\tau(\gamma)=\zeta_3\gamma$, where $\gamma^3$ is in $\mathrm{res}(K)$ if $L_1/K$ is unramified, or $\gamma$ is a uniformizer of $\mathcal{O}_{L_1}$ if $L_1/K$ is totally ramified. Then an easy computation shows that $a$ can be taken as stated in the proposition. 

    Then we compute the value of the $\mathbf{v}$ function. If $L_1/K$ is unramified, then the ramification groups are determined by the $C_2^2$ part as in the proof of Proposition \ref{v funtion for C22}, where $v_{L_1}(a)=v_{L_1}(\tau(a))=v_{L_1}(\tau^2(a))=-j$, hence $\mathbf{v}(L)=\frac{3}{2}(j+1)$; if $L_1/K$ is ramified, then the ramification groups are
    $$G_0=A_4, G_1=\dots=G_j=C_2^2,G_{j+1}=\{e\},$$
    and then by Theorem \ref{artin conductor}, $\mathbf{v}(L)=\frac{3}{2}+\frac{1}{2}j=\frac{1}{2}(j+3)$.
\end{proof}

\begin{prop}
    $f_{C_2^2}=10q^3+4q^2,f_{A_4}=32q^3+32q^2.$
\end{prop}

\begin{proof}
    For $L\in A(C_2^2)$, it is determined by $a,b$ as stated in Proposition \ref{v funtion for C22}. However, considering all the three nontrivial intermediate fields of $L/K$, we know that any two elements from $\{a,b,a+b\}$ give the same extension. We may furthermore assume that $v_L(a)\geqslant v_L(b)= v_L(a+b)$. On the other hand, for each extension, there are $3!=6$ ways to equip it with a $C_2^2$-action by mapping nontrivial elements in $C_2^2$ to nontrivial elements of $\mathrm{Gal}(L/K)$.

    Therefore, there are
    $$\frac{6\times 2q^{\frac{k-1}{2}}(q-1)}{2}$$
    choices for $L\in A(C_2^2)$ with $j=0,k>0$,
    $$\frac{6\times (2q^{\frac{j-1}{2}}(q-1))(2q^{\frac{j-1}{2}}(q-2))}{6}$$
    choices for $L\in A(C_2^2)$ with $j=k>0$, and
    $$\frac{6\times (2q^{\frac{j-1}{2}}(q-1))(2q^{\frac{k-1}{2}}(q-1))}{2}$$
    choices for $L\in A(C_2^2)$ with $0<j<k$.

    Consequently,
    \begin{align*}
        f_{C_2^2} 
        &= \sum_{k>0,(2,k)=1} 6q^{\frac{k-1}{2}}(q-1)q^{4-(k+1)}\\
        &+ \sum_{j>0,(2,j)=1} (2q^{\frac{j-1}{2}}(q-1))(2q^{\frac{j-1}{2}}(q-2))q^{4-\frac{3}{2}(j+1)}\\
        &+ \sum_{j>0,(2,j)=1}\sum_{k>j,(2,k)=1} 3(2q^{\frac{j-1}{2}}(q-1))(2q^{\frac{k-1}{2}}(q-1))q^{4-\frac{1}{2}(j+1)-(k+1)}\\
        &= 6q^4(q-1)\sum_{i=0}^\infty q^{i}q^{-2i-2}\\
        &+ 4(q-1)(q-2)q^4\sum_{i=0}^\infty q^i q^i q^{-3i-3}\\
        &+ 12(q-1)^2 q^4 \sum_{r=0}^\infty \sum_{s=1}^\infty q^r q^{r+s} q^{-r-1-2r-2s-2}\\
        &= 6q^2(q-1)\frac{1}{1-q^{-1}}+4(q-1)(q-2)q\frac{1}{1-q^{-1}}\\
        &+12(q-1)^2 q \frac{1}{1-q^{-1}}\frac{q^{-1}}{1-q^{-1}}\\
        &= 6q^3+4q^2(q-2)+12q^2=10q^3+4q^2.
    \end{align*}

    We compute $f_{A_4}$ in the same way. For $L\in A(A_4)$, it is determined by $L_1\in A(C_3)$ and $a\in L_1$ as in Proposition \ref{v function for A4}. Since $L_1/K$ is a Kummer extension, we have $1$ choice for the unramified one and $3$ choices for the totally ramified ones (ignoring the $C_3$-action on it). By abuse of notation, let $\tau$ be an element of $\mathrm{Gal}(L/K)$ of order $3$, whose restriction on $L_1$ generates $\mathrm{Gal}(L_1/K)$. For $L/L_1$, on one hand, $a,\tau(a),\tau^2(a)$ give the same extension; on the other hand, there are $3!=6$ ways to equip the extension with a $C_2^2$-action. Finally, whenever the $C_2^2$-extension is determined, one can choose an element of order $3$ in $A_4$ from a specific conjugacy class such that it acts on $L$ as $\tau$, which contains $4$ choices.

    Therefore, there are
    $$\frac{6\times 4\times (q^2)^{\frac{j-1}{2}}(q^2-1)}{3}$$
    choices for $L\in A(A_4)$ with $L_1$ unramified, $v_{L_1}(a)=-j$, and
    $$3\times \frac{6\times 4\times q^{j-1-\floor*{\frac{j-1}{2}}-\floor*{\frac{j-1}{3}}+\floor*{\frac{j-1}{6}}}(q-1)}{3}$$
    choices for $L\in A(A_4)$ with $L_1$ totally ramified, $v_{L_1}(a)=-j$.

    Consequently, 
    \begin{align*}
        f_{A_4}
        &= \sum_{j>0,(2,j)=1} 8q^{j-1}(q^2-1)q^{4-\frac{3}{2}(j+1)}\\
        &+ \sum_{j>0,(6,j)=1} 24q^{j-1-\floor*{\frac{j-1}{2}}-\floor*{\frac{j-1}{3}}+\floor*{\frac{j-1}{6}}}(q-1)q^{4-\frac{1}{2}(j+3)}\\
        &= 8(q^2-1)q^3\sum_{i=0}^\infty q^{2i+1}q^{-3i-3}\\
        &+ 24(q-1)q^3\sum_{r=0}^\infty (q^{2r+1}q^{-3r-2}+q^{2r+2}q^{-3r-4})\\
        &= 8(q^2-1)q\frac{1}{1-q^{-1}}+24(q-1)q(\frac{q}{1-q^{-1}}+\frac{1}{1-q^{-1}})\\
        &= 8(q+1)q^2+24(q^3+q^2)= 32q^3+32q^2.
    \end{align*}
\end{proof}

\begin{cor}\label{examples infinite}
    If the quotient singularity $X_1:=\AAA^4/C_2^2$ ($resp.$ $X_2:=\AAA^4/A_4$) has a crepant resolution $Y_1\to X_1$ ($resp.$ $Y_2\to X_2$), then $e(Y_1)=6$ ($resp.$ $e(Y_2)=10$). 
\end{cor}

\begin{proof}
\begin{align*}
    S(F_{C_2^2}) &=\frac{1}{4}(S(f_{\{e\}})+3S(f_{C_2})+S(f_{C_2^2}))=\frac{1}{4}(1+9+14)=6,\\
    S(F_{A_4}) &=\frac{S(f_{\{e\})})}{12}+\frac{S(f_{C_2})}{4}+\frac{S(f_{C_3})}{3}+\frac{S(f_{C_2^2})}{12}+\frac{S(f_{A_4})}{12}\\
    &= \frac{1}{12}+\frac{3}{4}+\frac{8}{3}+\frac{14}{12}+\frac{64}{12}=10.
\end{align*}
    
\end{proof}

\begin{rem}
    Although the author does not know if $X_1$ has a crepant resolution, $X_2$ does have a crepant resolution with Euler characteristic $10$, as constructed in \cite{fan2023crepant}.
\end{rem}

\begin{rem}
    We can furthermore compute that $F_{C_2^2}=q^4+4q^3+q^2$ and $F_{A_4}=q^4+6q^3+3q^2$, which can be seen as a realization of the stringy motives for quotient singularities $X_1$ and $X_2$ over $\mathbb{F}_q$, from the perspective of the motivic wild McKay correspondence. This coincides with the construction in \cite{fan2023crepant}, where the class of the crepant resolution of $\mathbb{A}^4/A_4$ in the (modified) Grothendieck ring is $\mathbb{L}^4+6\mathbb{L}^3+3\mathbb{L}^2$.
\end{rem}

\section*{Acknowledgements}
    The author expresses his gratitude to Prof. Takehiko Yasuda, Prof. Dr. Christian Liedtke, and Dr. Rin Gotou for valuable discussions. The author thanks his supervisor, Prof. Yukari Ito, for her advice and consideration. The author also thanks the anonymous referees for their helpful comments. This work was supported by JSPS KAKENHI Grant Number JP25KJ1153.
    



\bibliographystyle{elsarticle-harv} 
\bibliography{HCp}

\begin{thebibliography}{19}
\expandafter\ifx\csname natexlab\endcsname\relax\def\natexlab#1{#1}\fi
\providecommand{\url}[1]{\texttt{#1}}
\providecommand{\href}[2]{#2}
\providecommand{\path}[1]{#1}
\providecommand{\DOIprefix}{doi:}
\providecommand{\ArXivprefix}{arXiv:}
\providecommand{\URLprefix}{URL: }
\providecommand{\Pubmedprefix}{pmid:}
\providecommand{\doi}[1]{\href{http://dx.doi.org/#1}{\path{#1}}}
\providecommand{\Pubmed}[1]{\href{pmid:#1}{\path{#1}}}
\providecommand{\bibinfo}[2]{#2}
\ifx\xfnm\relax \def\xfnm[#1]{\unskip,\space#1}\fi
\bibitem[{Batyrev(1999)}]{batyrev1999non}
\bibinfo{author}{Batyrev, V.V.}, \bibinfo{year}{1999}.
\newblock \bibinfo{title}{{Non-Archimedean} integrals and stringy {Euler} numbers of log-terminal pairs}.
\newblock \bibinfo{journal}{Journal of the European Mathematical Society} \bibinfo{volume}{1}, \bibinfo{pages}{5--33}.
\bibitem[{Brion and Kumar(2005)}]{brion2005frobenius}
\bibinfo{author}{Brion, M.}, \bibinfo{author}{Kumar, S.}, \bibinfo{year}{2005}.
\newblock \bibinfo{title}{Frobenius splitting methods in geometry and representation theory}.
\newblock \bibinfo{publisher}{Springer}.
\bibitem[{Chen et~al.(2020)Chen, Du and Gao}]{chen2020modular}
\bibinfo{author}{Chen, Y.}, \bibinfo{author}{Du, R.}, \bibinfo{author}{Gao, Y.}, \bibinfo{year}{2020}.
\newblock \bibinfo{title}{Modular quotient varieties and singularities by the cyclic group of order $2p$}.
\newblock \bibinfo{journal}{Communications in Algebra} \bibinfo{volume}{48}, \bibinfo{pages}{5490--5500}.
\bibitem[{Fan(2022)}]{fanmaster}
\bibinfo{author}{Fan, L.}, \bibinfo{year}{2022}.
\newblock \bibinfo{title}{Crepant resolution of quotient singularities in positive characteristic}.
\newblock \bibinfo{type}{Master's thesis}. The University of Tokyo.
\bibitem[{Fan(2023)}]{fan2023crepant}
\bibinfo{author}{Fan, L.}, \bibinfo{year}{2023}.
\newblock \bibinfo{title}{Crepant resolution of $\mathbb{A}^4/{A}_4$ in characteristic $2$}.
\newblock \bibinfo{journal}{Proceedings of the Japan Academy, Series A: Mathematical Sciences} \bibinfo{volume}{99}, \bibinfo{pages}{71--76}.
\bibitem[{Hayashi et~al.(2017)Hayashi, Ito and Sekiya}]{hayashi2017existence}
\bibinfo{author}{Hayashi, T.}, \bibinfo{author}{Ito, Y.}, \bibinfo{author}{Sekiya, Y.}, \bibinfo{year}{2017}.
\newblock \bibinfo{title}{Existence of crepant resolutions}, in: \bibinfo{booktitle}{{Higher Dimensional Algebraic Geometry: In honour of Professor Yujiro Kawamata's sixtieth birthday}}. volume~\bibinfo{volume}{74} of \textit{\bibinfo{series}{Advanced Studies in Pure Mathematics}}, pp. \bibinfo{pages}{185--202}.
\bibitem[{Higman(1954)}]{higman1954indecomposable}
\bibinfo{author}{Higman, D.}, \bibinfo{year}{1954}.
\newblock \bibinfo{title}{Indecomposable representations at characteristic $p$}.
\newblock \bibinfo{journal}{Duke Mathematical Journal} \bibinfo{volume}{21}, \bibinfo{pages}{377--381}.
\bibitem[{Ito(1995)}]{ito1994crepant}
\bibinfo{author}{Ito, Y.}, \bibinfo{year}{1995}.
\newblock \bibinfo{title}{Crepant resolution of trihedral singularities and the orbifold {E}uler characteristics}.
\newblock \bibinfo{journal}{International Journal of Mathematics} \bibinfo{volume}{6}, \bibinfo{pages}{33--44}.
\bibitem[{Ito and Reid(1996)}]{ito1996mckay}
\bibinfo{author}{Ito, Y.}, \bibinfo{author}{Reid, M.}, \bibinfo{year}{1996}.
\newblock \bibinfo{title}{The {M}c{K}ay correspondence for finite subgroups of $\mathrm{SL}(3,\mathbb{C})$}.
\newblock \bibinfo{journal}{Higher-dimensional complex varieties (Trento, 1994)} , \bibinfo{pages}{221--240}.
\bibitem[{Serre(2013)}]{serre2013local}
\bibinfo{author}{Serre, J.P.}, \bibinfo{year}{2013}.
\newblock \bibinfo{title}{Local fields}. volume~\bibinfo{volume}{67}.
\newblock \bibinfo{publisher}{Springer Science \& Business Media}.
\bibitem[{{The LMFDB Collaboration}(2024)}]{lmfdb:3.13.12.1}
\bibinfo{author}{{The LMFDB Collaboration}}, \bibinfo{year}{2024}.
\newblock \bibinfo{title}{The {L}-functions and modular forms database, {H}ome page of the $p$-adic field \texttt{3.13.12.1}}.
\newblock \bibinfo{howpublished}{\mbox{\url{https://www.lmfdb.org/padicField/3.13.12.1}}}.
\bibitem[{Webb(2016)}]{webb2016course}
\bibinfo{author}{Webb, P.}, \bibinfo{year}{2016}.
\newblock \bibinfo{title}{A course in finite group representation theory}. volume \bibinfo{volume}{161}.
\newblock \bibinfo{publisher}{Cambridge University Press}.
\bibitem[{Wood and Yasuda(2015)}]{wood2015mass}
\bibinfo{author}{Wood, M.M.}, \bibinfo{author}{Yasuda, T.}, \bibinfo{year}{2015}.
\newblock \bibinfo{title}{Mass formulas for local {Galois} representations and quotient singularities. {I}: a comparison of counting functions}.
\newblock \bibinfo{journal}{International Mathematics Research Notices} \bibinfo{volume}{2015}, \bibinfo{pages}{12590--12619}.
\bibitem[{Wu and Scheidler(2010)}]{wu2010ramification}
\bibinfo{author}{Wu, Q.}, \bibinfo{author}{Scheidler, R.}, \bibinfo{year}{2010}.
\newblock \bibinfo{title}{The ramification groups and different of a compositum of {A}rtin--{S}chreier extensions}.
\newblock \bibinfo{journal}{International Journal of Number Theory} \bibinfo{volume}{6}, \bibinfo{pages}{1541--1564}.
\bibitem[{Yamamoto(2021)}]{yamamoto2021crepant}
\bibinfo{author}{Yamamoto, T.}, \bibinfo{year}{2021}.
\newblock \bibinfo{title}{Crepant resolutions of quotient varieties in positive characteristics and their {Euler} characteristics}.
\newblock \bibinfo{journal}{arXiv:2106.11526} .
\bibitem[{Yamamoto(2022)}]{yamamoto2022mass}
\bibinfo{author}{Yamamoto, T.}, \bibinfo{year}{2022}.
\newblock \bibinfo{title}{Mass formulas and stringy point-count for semi-direct products of tame abelian groups and wild symmetric or cyclic groups in characteristic three}.
\newblock \bibinfo{journal}{arXiv:2203.14440} .
\bibitem[{Yasuda(2014)}]{yasuda2014cyclic}
\bibinfo{author}{Yasuda, T.}, \bibinfo{year}{2014}.
\newblock \bibinfo{title}{The $p$-cyclic {McKay} correspondence via motivic integration}.
\newblock \bibinfo{journal}{Compositio Mathematica} \bibinfo{volume}{150}, \bibinfo{pages}{1125--1168}.
\bibitem[{Yasuda(2017a)}]{yasuda2017toward}
\bibinfo{author}{Yasuda, T.}, \bibinfo{year}{2017}a.
\newblock \bibinfo{title}{Toward motivic integration over wild {D}eligne-{M}umford stacks}, in: \bibinfo{booktitle}{{Higher Dimensional Algebraic Geometry: In honour of Professor Yujiro Kawamata's sixtieth birthday}}. volume~\bibinfo{volume}{74} of \textit{\bibinfo{series}{Advanced Studies in Pure Mathematics}}, pp. \bibinfo{pages}{407--438}.
\bibitem[{Yasuda(2017b)}]{yasuda2017wild}
\bibinfo{author}{Yasuda, T.}, \bibinfo{year}{2017}b.
\newblock \bibinfo{title}{The wild {McKay} correspondence and $p$-adic measures}.
\newblock \bibinfo{journal}{Journal of the European Mathematical Society (EMS Publishing)} \bibinfo{volume}{19}, \bibinfo{pages}{3709--3734}.

\end{thebibliography}






\end{document}